\RequirePackage{fix-cm}
\documentclass[smallextended]{svjour3}       
\smartqed  

\usepackage{multirow,mathtools}
\usepackage{amsmath,amssymb}
\usepackage{amsfonts}
\usepackage{url}
\usepackage[pointedenum]{paralist}
\usepackage{cases}
\usepackage[breaklinks=true,pdfstartview=FitH]{hyperref}

\newcommand{\sjwsolved}[1]{}

\newcommand{\dist}{\mbox{\rm dist}}
\newcommand{\bx}{\bar{x}}

\newcommand{\bxT}{\bar{x}_T}

\newcommand{\g}{\nabla f(x)}
\newcommand{\gk}{\nabla f(x_k)}
\newcommand{\gT}{\nabla f(x_T)}

\newcommand{\nabs}[1]{{\left\|{#1}\right\|}}
\newcommand{\E}{\mathbb{E}}
\newcommand{\pmin}{p_{\min}}

\newcommand{\TheTitle}{First-Order Algorithms
Converge Faster than $O(1/k)$ on Convex Problems}

\title{{\TheTitle} \thanks{This work was supported by NSF Awards
    IIS-1447449, 1628384, 1634597, and 1740707; Subcontract 8F-30039
    from Argonne National Laboratory; and Award N660011824020 from the
    DARPA Lagrange Program.} }

\author{Ching-pei Lee \and Stephen~J. Wright}

\institute{Department of Computer Sciences and Wisconsin Institute for
	Discovery, University of
	Wisconsin-Madison, Madison, WI \\
	\email{\{ching-pei,swright\}@cs.wisc.edu}
}
\date{}
\begin{document}

\maketitle

\begin{abstract}
  It is well known that both gradient descent and
  stochastic coordinate descent achieve a global convergence rate of
  $O(1/k)$ in the objective value, when applied to a scheme for
  minimizing a Lipschitz-continuously differentiable, unconstrained
  convex function.  In this work, we improve this rate to $o(1/k)$.
  We extend the result to proximal gradient and proximal coordinate
  descent on regularized problems to show similar $o(1/k)$ convergence
  rates. The result is tight in the sense that a rate of
  $O(1/k^{1+\epsilon})$ is not generally attainable for any
  $\epsilon>0$, for any of these methods.

\end{abstract}

\section{Introduction} \label{sec:intro}
Consider the unconstrained optimization problem
\begin{equation}
\min_x \, f(x),
\label{eq:f}
\end{equation}
where $f$ has domain in an inner-product space and is convex and $L$-Lipschitz
continuously differentiable for some $L > 0$. We assume throughout
that the solution set $\Omega$ is non-empty. (Elementary arguments
based on the convexity and continuity of $f$ show that $\Omega$ is a
closed convex set.) Classical convergence theory for gradient descent
on this problem indicates a $O(1/k)$ global convergence
rate in the function value. Specifically, if
\begin{equation}
x_{k+1} \coloneqq x_k - \alpha_k \gk, \quad k=0,1,2,\dotsc,
\label{eq:gd}
\end{equation}
and $\alpha_k \equiv \bar\alpha \in (0, 1/L]$, we have
\begin{equation}
f\left( x_k \right) - f^* \leq \frac{\dist(x_0, \Omega)^2}{2\bar\alpha
k}, \quad k=1,2,\dotsc,
\label{eq:1k}
\end{equation}
where $f^*$ is the optimal objective value and $\dist(x, \Omega)$
denotes the distance from $x$ to the solution set.  The proof of
\eqref{eq:1k} relies on showing that
\begin{equation}
	k \left(f\left( x_k \right) - f^*\right) \leq \sum_{T=1}^{k}
	\left(f\left( x_{T} \right) - f^* \right) \leq \frac{1}{2\bar
	\alpha} \dist(x_0, \Omega)^2, \quad k=1,2,\dotsc,
	\label{eq:sum}
\end{equation}
where the first inequality utilizes the fact that gradient descent is
a descent method (yielding a nonincreasing sequence of function values
$\{ f(x_k \}$).
We demonstrate in this paper that the bound \eqref{eq:1k} is not
tight, in the sense that $k (f(x_k)-f^*) \to 0$, and thus $f(x_k)-f^*
= o(1/k)$. This result is a consequence of the following technical
lemma.
\begin{lemma}
\label{lemma:techniques}
Let $\{\Delta_k\}$ be a nonnegative sequence satisfying the following
conditions:
\begin{enumerate}
\item $\{\Delta_k\}$ is monotonically decreasing;
\item $\{\Delta_k\}$ is summable, that is, $\sum_{k=0}^{\infty}
  \Delta_k < \infty$.
\end{enumerate}
Then $k \Delta_k \to 0$, so that $\Delta_k = o(1/k)$.
\end{lemma}
\begin{proof}
  The proof uses simplified elements of the proofs of Lemmas~2 and 9
  of Section~2.2.1 from \cite{Pol87a}.  Define $s_k \coloneqq k \Delta_k$ and
  $u_k \coloneqq s_k + \sum_{i=k}^{\infty} \Delta_i$.  Note that
  \begin{equation} \label{eq:vp1}
  s_{k+1} = (k+1) \Delta_{k+1} \le k \Delta_k + \Delta_{k+1} \leq s_k
  + \Delta_k.
  \end{equation}
  From \eqref{eq:vp1} we have
  \[
  u_{k+1} = s_{k+1} + \sum_{i=k+1}^{\infty} \Delta_i \le s_k +
  \Delta_k + \sum_{i=k+1}^{\infty} \Delta_i = s_k +
  \sum_{i=k}^{\infty} \Delta_i = u_k,
  \]
  so that $\{ u_k \}$ is a monotonically decreasing nonnegative
  sequence. Thus there is $u \ge 0$ such that $u_k \to u$, and since
  $\lim_{k \to \infty} \sum_{i=k}^{\infty} \Delta_i = 0$, we have $s_k
  \to u$ also.

  Assuming for contradiction that $u>0$, there exists $k_0>0$ such
  that $s_k \ge u/2>0$ for all $k \ge k_0$, so that $\Delta_k \ge
  {u}/{(2k)}$ for all $k \ge k_0$. This contradicts the summability of
  $\{ \Delta_k \}$.  Therefore we have $u=0$, so that $k \Delta_k =
  s_k \to 0$, proving the result.
  \qed
\end{proof}


Our claim about the fixed-step gradient descent method follows
immediately by setting $\Delta_k = f(x_k)-f^*$ in
Lemma~\ref{lemma:techniques}. We state the result formally as follows,
and prove it at the start of Section~\ref{sec:main}.
\begin{theorem}
\label{thm:main}
Consider \eqref{eq:f} with $f$ convex and $L$-Lipschitz continuously
differentiable and nonempty solution set $\Omega$.  If the step sizes
satisfy $\alpha_k \equiv \bar\alpha \in (0, 1 / L]$ for all $k$, then
  gradient descent \eqref{eq:gd} generates objective values $f(x_k)$
  that converge to $f^*$ at an asymptotic rate of $o(1/k)$.
\end{theorem}

This result shows that the $o(1/k)$ rate for gradient descent with a
fixed short step size is universal on convex problems, without any
additional requirements such as the boundedness of $\Omega$ assumed in
\cite[Proposition~1.3.3]{Ber16a}.  In the remainder of the paper, we
show that this faster rate holds for several other smooth optimization
algorithms, including gradient descent with fixed steps in the larger
range $(0,2/L)$, gradient descent with various line-search strategies,
and stochastic coordinate descent with arbitrary sampling
strategies.  We then extend the result to algorithms for regularized
convex optimization problems, including proximal gradient and
stochastic proximal coordinate descent.

Except for the cases of coordinate descent and proximal coordinate
descent which require a finite-dimensional space so that all the
coordinates can be processed, our results apply to any inner-product
spaces.  Assumptions such as bounded solution set, bounded level set,
or bounded distance to the solution set, which are commonly assumed in
the literature, are all unnecessary.  We can remove these assumptions
because an implicit regularization property causes the iterates to
stay within a bounded area.

In our description, the Euclidean norm is used for simplicity, but our
results can be extended directly to any norms induced by an inner
product,\footnote{We meant that given an inner product
	$<\cdot,\cdot>$, the norm $\|\cdot\|$ is defined as $\|x\|
	\coloneqq \sqrt{<x,x>}$.}
provided that Lipschitz continuity of $\nabla f$ is defined
with respect to the corresponding norm and its dual norm.

\paragraph{Related Work.}
Our work was inspired by \cite[Corollary~2]{PenZZ18a} and
\cite[Proposition~1.3.3]{Ber16a}, which improve convergence for
certain algorithms and problems on convex problems in a Euclidean
space from $O(1/k)$ to $o(1/k)$ when the level set is compact.  Our
paper develops improved convergence rates of several algorithms on
convex problems without the assumption on the level set, with most of
our results applying to non-Euclidean Hilbert spaces.  The main proof
techniques in this work are somewhat different from those in the works
cited here.

For an accelerated version of proximal gradient on convex problems,
it is proved in \cite{AttP16a} that the convergence rate can be
improved from $O(1/k^2)$ to $o(1/k^2)$.  Accelerated proximal
gradient is a more complicated algorithm than the nonaccelerated
versions we discuss, and thus \cite{AttP16a} require a more
complicated analysis that is quite different from ours.

\cite{DenLPY17a} have stated a version of
Lemma~\ref{lemma:techniques} with a proof different from the proof
that we present, using it to show the
convergence rate of the quantity $\|x_k - x_{k+1}\|$ of a version of
the alternating-directions method of multipliers (ADMM).  Our work
differs in the range of algorithms considered and the nature of the
convergence. We also provide a discussion of the tightness of the
$o(1/k)$ convergence rate.

\section{Main Results on Unconstrained Smooth Problems}
\label{sec:main}

We start by detailing the procedure for obtaining \eqref{eq:sum}, to
complete the proof of Theorem~\ref{thm:main}.  First, we define
\begin{equation}
M(\alpha) \coloneqq \alpha - \tfrac12 L\alpha^2.
\label{eq:M}
\end{equation}
From the Lipschitz continuity of $\nabla f$, we have for any
point $x$ and any real number $\alpha$ that
\begin{align}
f\left( x - \alpha \g \right)
\leq f( x ) - \g^\top \left( \alpha \g  \right)
	+ \frac{L}{2}\nabs{\alpha \g}^2
\label{eq:Lip}
= f( x) - M( \alpha ) \nabs{\g}^2.
\end{align}
Clearly,
\begin{equation}
\alpha \in \left(0,\frac1L\right] \quad \Rightarrow \quad
M(\alpha) \geq \tfrac12 {\alpha}> 0,
\label{eq:Mbound}
\end{equation}
so in this case, we have by rearranging \eqref{eq:Lip} that
\begin{align}
  \| \nabla f(x) \|^2  \le \frac{1}{M(\alpha)} \left( f(x) -
  f(x-\alpha \g) \right) \le \frac{2}{\alpha} \left( f(x) - f(x-\alpha
  \g) \right).
   \label{eq:th2}
  \end{align}

Considering any solution $\bx \in \Omega$ and any $T \geq 0$, we have
for gradient descent \eqref{eq:gd} that
\begin{align}
\label{eq:dist}
\nabs{x_{T+1} - \bx}^2 = \nabs{x_T - \alpha_T \gT - \bx}^2
= \nabs{x_T - \bx}^2 + \alpha_T^2 \|\gT\|^2 - 2 \alpha_T \gT^\top
\left( x_T - \bx \right).
\end{align}
Since $\alpha_T \in (0, 1/L]$ in \eqref{eq:dist}, from
  \eqref{eq:th2} and the convexity of $f$ (implying $\nabla
  f(x_T)^T(\bx-x_T) \le f^* - f(x_T)$), we have
\begin{equation}
\nabs{x_{T+1} - \bx}^2 \leq \nabs{x_T - \bx}^2 + 2\alpha_T \left(
	f\left( x_{T} \right) - f\left( x_{T+1} \right)\right) + 2 \alpha_T
	\left(f^*- f\left( x_T \right) \right).
\label{eq:rk}
\end{equation}
By rearranging \eqref{eq:rk} and using $\alpha_T  \equiv \bar\alpha \in (0,1/L]$,
\begin{align}
f\left( x_{T+1} \right) - f^*
\leq \frac{1}{2\bar\alpha} \left(
\nabs{x_T - \bx}^2 - \nabs{x_{T+1} - \bx}^2\right).
\label{eq:tosum}
\end{align}
We then obtain \eqref{eq:sum} by summing \eqref{eq:tosum} from $T=0$
to $T=k-1$ and noticing that $\bx$ is arbitrary in $\Omega$.

Theorem~\ref{thm:main} applies to step sizes in the range $(0,1/L]$
  only, but it is known that gradient descent converges at the rate of
  $O(1/k)$ for both the fixed step size scheme with $\bar\alpha \in
  (0,2/L)$ and line-search schemes.  Next, we show that $o(1/k)$ rates
  hold for these variants too.  We then extend the result to
  stochastic coordinate descent with arbitrary sampling of
  coordinates.

\subsection{Gradient Descent with Longer Steps}
\label{subsec:linesearch}
In this subsection, we allow the steplengths $\alpha_k$ for
\eqref{eq:gd} to vary from iteration to iteration, according to the
following conditions, for some $\gamma \in (0,1]$:
\begin{subequations}
\begin{gather}
\label{eq:bdd}
\alpha_k \in [C_2, C_1],\quad C_2 \in \left(0, \frac{ 2 - \gamma
}{L}\right],\quad C_1 \geq C_2,\\
		\label{eq:sufficient}
		f\left( x_k - \alpha_k \gk \right) \leq f\left( x_k \right) -
		\frac{\gamma \alpha_k}{2} \nabs{\gk}^2,
\end{gather}
\label{eq:linesearch}
\end{subequations}
Note that these conditions encompass a fixed-steplength strategy with
$\alpha_k \equiv C_2$ as a special case, by setting $C_1 = C_2$, and
noting that condition \eqref{eq:sufficient} is a consequence of
\eqref{eq:Lip}. (Note too that $\alpha_k \equiv C_2 \in (0,
( 2 - \gamma) / L]$ can be almost
twice as large as the bound $1/L$ considered above.)

The main result for this subsection is as follows.
\begin{theorem}
\label{thm:main1}
Consider \eqref{eq:f} with $f$ convex and $L$-Lipschitz continuously
differentiable and nonempty solution set $\Omega$.  If the step sizes
$\alpha_k$ satisfy \eqref{eq:linesearch}, then gradient descent
\eqref{eq:gd} generates objective values $f(x_k)$ converging to $f^*$
at an asymptotic rate of $o(1/k)$.
\end{theorem}

We give two alternative proofs of this result to provide different
insights.  The first proof is similar to the one we presented for
Theorem \ref{thm:main} at the start of this section.  The second proof
holds only for Euclidean spaces. This proof improves the standard
proof of \cite[Section~2.1.5]{Nes04a}.

We start from the following lemma, which verifies that the iterates
remain in a bounded set and is used in both proofs.

\begin{lemma}
\label{lemma:rk}
Consider algorithm \eqref{eq:gd} with any initial point $x_0$, and
assume that $f$ is convex and $L$-Lipschitz-continuously
differentiable for some $L > 0$.  Then when the sequence of
steplengths $\alpha_k$ is chosen to satisfy \eqref{eq:linesearch}, all
iterates $x_k$ lie in a bounded set.  In particular, for any $\bx \in
\Omega$ and any $k \geq 0$, we have that
\begin{align}
\label{eq:toprove1}
\|x_{k+1} - \bx\|^2 &\le \|x_0 - \bx\|^2 +
\frac{2 C_1}{\gamma}\left( f\left( x_0 \right) - f\left( x_{k+1}
\right) \right)
+ 2 C_2 \sum_{T=0}^k\left( f^* - f\left( x_T \right)
\right)\\
&\le \|x_0 - \bx\|^2 + \frac{2 C_1}{\gamma}\left( f\left( x_0 \right)
- f^* \right).
\label{eq:toprove2}
\end{align}
\end{lemma}
\begin{proof}
	By \eqref{eq:sufficient} and the convexity of $f$, \eqref{eq:dist}
	further implies that for any $T \ge 0$,
\begin{align}
\label{eq:bound}
\nabs{x_{T+1} - \bx}^2 - \nabs{x_{T} - \bx}^2
\le  \frac{2\alpha_T}{\gamma}\left( f\left( x_T \right) - f\left(
x_{T+1}
\right)\right) + 2 \alpha_T \left( f^* - f\left( x_T \right) \right).
\end{align}
We know that the first term is nonnegative from \eqref{eq:sufficient},
while the second term is nonpositive from the optimality of $f^*$.
Therefore, \eqref{eq:bound} implies
\begin{align}
\label{eq:bound2}
\nabs{x_{T+1} - \bx}^2 - \nabs{x_{T} - \bx}^2
\le  \frac{2 C_1}{\gamma}\left( f\left( x_T \right) - f\left(
x_{T+1}
\right)\right) + 2 C_2 \left( f^* - f\left( x_T \right) \right).
\end{align}
We then obtain \eqref{eq:toprove1}
by summing \eqref{eq:bound2} for $T=0,1,\dotsc,k$ and telescoping.
By noting that $f(x_k) \geq f^*$ for all $k$,
\eqref{eq:toprove2} follows.
\qed
\end{proof}

The first proof of Theorem \ref{thm:main1} is as follows.

\begin{proof}[First Proof of Theorem \ref{thm:main1}]
We again consider Lemma~\ref{lemma:techniques} with $\Delta_k
\coloneqq f(x_k) - f^*$, which is always nonnegative from the
optimality of $f^*$.  Monotonicity is clear from
\eqref{eq:sufficient}, so we just need to show summability.  By
rearranging \eqref{eq:toprove1} and noting $f(x_{k+1}) \geq f^*$, we
obtain
\begin{align*}
2C_2\sum_{T=0}^k \Delta_T \le \nabs{x_0 - \bx}^2 - \nabs{x_{k+1} - \bx}^2 +
	\frac{2 C_1}{\gamma}\Delta_0
\le \nabs{x_0 - \bx}^2+ \frac{2 C_1}{\gamma}\Delta_0.
\tag*{\qed}
\end{align*}
\end{proof}

For the second proof of Theorem~\ref{thm:main1}, we first outline the
analysis from \cite[Section~2.1.5]{Nes04a} and then show how it can
be modified to produce the desired $o(1/k)$ rate. Denote by
$\bxT$ the projection of $x_T$ onto $\Omega$ (which is well defined
because $\Omega$ is nonempty, closed, and convex). We can utilize the
convexity of $f$ to obtain
\begin{equation*}
\Delta_T \leq \gT^\top \left( x_T - \bxT \right) \leq \|\gT\| \dist\left(
	x_T,\Omega \right),
\end{equation*}
so that
\begin{equation} \label{eq:r0bound}
	\| \nabla f(x_T) \| \ge \frac{\Delta_T}{\dist ( x_T,\Omega)}.
\end{equation}
By subtracting $f^*$ from both sides of
  \eqref{eq:sufficient} and using $\alpha_k \ge C_2$ and
  \eqref{eq:r0bound}, we obtain 
  \[
\Delta_{T+1} \leq \Delta_T -
\frac{C_2 \gamma\Delta_T^2}{2 \dist\left( x_T,\Omega \right)^2}.
\]
By dividing both sides of this expression by $\Delta_T \Delta_{T+1}$
and using $\Delta_{T+1} \le \Delta_T$, we obtain
\begin{equation}
\frac{1}{\Delta_{T+1}} \geq \frac{1}{\Delta_T} + \frac{C_2 \gamma
\Delta_T }{2 \dist\left( x_T, \Omega \right)^2 \Delta_{T+1}}
\ge
\frac{1}{\Delta_T} + \frac{C_2 \gamma}{2 \dist\left( x_T, \Omega \right)^2}.
\label{eq:rate}
\end{equation}
By summing \eqref{eq:rate} over $T=0,1,\dotsc,k-1$,  we obtain
\begin{equation}
\frac{1}{\Delta_{k}} \geq \frac{1}{\Delta_0} +
\sum_{T=0}^{k-1}
\frac{C_2 \gamma}{2 \dist\left( x_T,\Omega \right)^2}
\quad
	\Rightarrow
	\quad
	\Delta_{k} \leq \frac{1}{\sum_{T=0}^{k-1}
	\frac{C_2 \gamma}{2 \dist\left( x_T,\Omega \right)^2}}.
\label{eq:fromhere}
\end{equation}
A $O(1/k)$ rate is obtained by noting from Lemma~\ref{lemma:rk} that
$\dist(x_T, \Omega) \le R_0$ for some $R_0>0$ and all $T$, so that
\begin{equation}
	\sum_{T=0}^{k-1} \frac{1}{\dist\left( x_T,\Omega \right)^2} \geq
	\frac{k}{R_0^2}.
\label{eq:toimprove}
\end{equation}

Our alternative proof uses the fact that \eqref{eq:toimprove} is a
loose bound for Euclidean spaces and that an improved result can be
obtained by working directly with \eqref{eq:fromhere}.  We first use
the Bolzano-Weierstrass theorem (a bounded and closed set is
sequentially compact in a Euclidean space) together with
Lemma~\ref{lemma:rk}, to show that the sequence $\{x_k \}$ approaches
the solution set $\Omega$.
\begin{lemma}
\label{lemma:conv}
Assume the conditions in Lemma~\ref{lemma:rk} and in addition that $f$
has domain in a Euclidean space $f: \Re^n \rightarrow \Re$.  We have
\begin{equation}
\lim_{k \rightarrow \infty}\, \dist\left( x_k,\Omega \right) = 0.
\label{eq:conv}
\end{equation}
\end{lemma}
\begin{proof}
 The proof is similar to \cite[Proposition~1]{PenZZ18a}.  Assume for
 contradiction that \eqref{eq:conv} does not hold.  Then there are
 $\epsilon > 0$ and an infinite increasing sequence $\{k_i\}$,
 $i=1,2,\dotsc$, such that
\begin{equation}
	\dist\left( x_{k_i}, \Omega \right) \geq \epsilon, \quad i=1,2,\dotsc.
	\label{eq:contradict}
\end{equation}
From Lemma~\ref{lemma:rk} and that $\{x_{k_i}\} \subset \Re^n$, we can
the sequence $\{x_{k_i}\}$ lies in a compact set and therefore has an
accumulation point $x^*$. From \eqref{eq:rate}, we have
$$\frac{1}{\Delta_{k_{i+1}}} \ge \frac{1}{\Delta_{k_i}} +
\frac{C_2\gamma}{2\epsilon^2},$$ so
that $1/\Delta_k \uparrow \infty$ and hence $\Delta_k \downarrow
0$. By continuity of $f$, it follows that $f(x^*) = f^*$, so that $x^*
\in \Omega$ by definition, contradicting \eqref{eq:contradict}.
\qed
\end{proof}

We note that a result similar to Lemma \ref{lemma:conv} has been given
in \cite{BurGIS95a} using a more complicated argument with more
restricted choices of $\alpha$.


\begin{proof}[Second Proof of Theorem~\ref{thm:main1}, for
	Euclidean Spaces]
We start with \eqref{eq:fromhere} and show that
\begin{equation*}
\lim_{k\rightarrow \infty} \frac{\frac{1}{\frac{C_2 \gamma}{2}
\sum_{T=0}^{k-1}
\frac{1}{\dist(x_T,\Omega)^2}}}{\frac{1}{k}} = 0,
\end{equation*}
or, equivalently,
\begin{equation}
\lim_{k \rightarrow \infty} \frac{k}{ \sum_{T = 0}^{k-1}
\frac{1}{\dist(x_T,\Omega)^2}} = 0.
\label{eq:inf}
\end{equation}
From the arithmetic-mean / harmonic-mean inequality,\footnote{ This
  inequality says that for any real numbers $a_1,\dotsc, a_n > 0$,
  their harmonic mean does not exceed their arithmetic mean.  Namely,
\begin{equation*}
\frac{n}{\sum_{i=1}^na_i^{-1}} \leq \frac{\sum_{i=1}^n a_i}{n}.
\end{equation*}
}
we have that
\begin{equation}
0 \le \frac{k}{ \sum_{T = 0}^{k-1} \frac{1}{\dist(x_T,\Omega)^2}}
\leq \frac{\sum_{T=0}^{k-1} \dist(x_T,\Omega)^2}{k}.
\label{eq:amhm}
\end{equation}
Lemma \ref{lemma:conv} shows that $\dist(x_T, \Omega) \to 0$, so by
the Stolz-Ces\`aro theorem (see, for example, \cite{Mur09a}), the
right-hand side of \eqref{eq:amhm} converges to $0$.  Therefore, from
the sandwich lemma, \eqref{eq:inf} holds.  \qed
\end{proof}

\subsection{Coordinate Descent}
\label{subsec:cd}
We now extend Theorem~\ref{thm:main} to the case of randomized
coordinate descent.  Our results can extend immediately to
block-coordinate descent with fixed blocks.  Our analysis for
coordinate descent requires Euclidean spaces so that coordinate
descent can go through all coordinates.

The standard short-step coordinate descent procedure requires
knowledge of coordinate-wise Lipschitz constants.  Denoting by $e_i$
the $i$th unit vector, we denote by $L_i \geq 0$ the constants such
that:
\begin{equation}
	\left| \nabla_i f(x) - \nabla_i f(x + h e_i)\right | \leq L_i
	\left|h\right|, \quad 
	\mbox{for all $x \in \Re^n$ and all $h \in \Re$},
	\label{eq:Ls}
\end{equation}
where $\nabla_i f(\cdot)$ denotes the $i$th coordinate of the gradient.
Note that if $\nabla f(x)$ is $L$-Lipschitz continuous, there always
exist $L_1,\dotsc,L_n \in [0,L]$ such that \eqref{eq:Ls} holds.
Without loss of generality, we assume $L_i > 0$ for all
$i$.
Given parameters $\{\bar L_i\}_{i=1}^n$ such that $\bar L_i
\geq L_i$ for all $i$, the coordinate descent update is
\begin{equation}
	x_{k+1} \leftarrow x_k - \frac{\nabla_{i_k} f\left( x_k
\right)}{\bar L_{i_k}} e_{i_k},
	\label{eq:cd}
\end{equation}
where $i_k$ is the coordinate selected for updating at the $k$th iteration.
We consider the general case of stochastic coordinate
descent in which each $i_k$ is independently identically distributed
following a fixed prespecified probability distribution
$p_1,\dotsc,p_n$ satisfying
\begin{equation}
	p_i  \ge \pmin ,\quad  i=1,2,\dotsc,n; \quad \sum_{i=1}^n p_i = 1,
	\label{eq:prob}
\end{equation}
for some constant $\pmin>0$.  Nesterov \cite{Nes12a} proves that
stochastic coordinate descent has a $O(1/k)$ convergence rate (in
expectation of $f$) on convex problems.  We show below that this rate
can be improved to $o(1/k)$.
\begin{theorem}
\label{thm:cd}
Consider \eqref{eq:f} with $f$ convex and nonempty solution set
$\Omega$, and that componentwise-Lipschitz continuous differentiability
\eqref{eq:Ls} holds with some
$L_1,\dotsc,L_n > 0$.  If we apply coordinate descent \eqref{eq:cd}
and at each iteration, $i_k$ is independently picked at random
following a probability distribution satisfying \eqref{eq:prob}, then
the expected objective $\E_{i_0,i_1,\dotsc,i_{k-1}}[f(x_k)]$ converges to
$f^*$ at an asymptotic rate of $o(1/k)$.
\end{theorem}
\begin{proof}
From \eqref{eq:Ls} and that $\bar L_i \geq L_i$, by treating all other
coordinates as non-variables, we have that for any $T \geq 0$,
\begin{equation}
	f\left( x_T - \frac{\nabla_i f\left( x_T
	\right)}{\bar L_i} e_i \right) - f\left( x_T \right) \leq
	- \frac{1}{2 \bar L_i} \nabs{\nabla_i f\left( x_T \right)}^2,
	i=1,\dotsc,n,
	\label{eq:cdsuff}
\end{equation}
showing that the algorithm decreases $f$ at each iteration.  Consider
any $\bx \in \Omega$, by defining
\begin{equation}
r_T^2 \coloneqq \sum_{i=1}^n \frac{\bar L_i}{p_i} \nabs{\left( x_T - \bx
\right)_i}^2,
\label{eq:r}
\end{equation}
we have from \eqref{eq:cd} that
\begin{equation*}
r_{T+1}^2 = r_T^2 + \frac{1}{ \bar L_{i_T}p_{i_T}} \left\| \nabla_{i_T}
	f\left( x_T \right) \right\|^2 -  \frac{2}{p_{i_T}} \nabla_{i_T} f\left(
	x_T \right)^\top \left( x_T - \bx \right)_{i_T}.
\end{equation*}
By taking expectation over $i_T$ on both sides of the above expression, we
obtain from the convexity of $f$ and \eqref{eq:cdsuff} that
\begin{align}
\nonumber
\E_{i_T}\left[ r_{T+1}^2 \right] - r_T^2
\stackrel{\eqref{eq:cdsuff}}{\leq}&~ \frac{1}{\pmin}
	\sum_{i=1}^n 2 p_i \left( f\left( x_T\right) - f\left( x_T -
	\frac{\nabla_i f\left( x_T \right)}{\bar L_i} e_i \right) \right) - 2
	\nabla f\left( x_T \right)^\top \left( x_T - \bx \right)\\
\leq&~ \frac{2}{\pmin} \left(f \left( x_T \right) - \E_{i_T}
	\left[f\left( x_{T+1} \right)\right]  \right)+ 2 \left( f^* - f\left(
	x_T \right) \right).
\label{eq:tosumcd}
\end{align}
By taking expectation over $i_0,i_1,\dotsc,i_{T-1}$ on
\eqref{eq:tosumcd} and summing \eqref{eq:tosumcd} over $T=0,1,
\dotsc,k$, we obtain
\begin{align*}
	2 \sum_{T=0}^k \left(\E_{i_0,\dotsc,i_{T-1}}
	\left[f(x_T)\right] - f^*\right)
	&~\leq r_0^2 - \E_{i_0,\dotsc,i_k}\left[ r_{k+1}^2\right] + \frac{2 \left(
		f\left( x_{0} \right) - \E_{i_0,\dotsc,i_k}\left[
		f\left( x_{k+1} \right)\right] \right)}{\pmin}\\
&~\leq r_0^2 + \frac{2 \left(f\left( x_0 \right) - f^*\right)  }{\pmin}.
\end{align*}
The result now follows from Lemma~\ref{lemma:techniques}.
\qed
\end{proof}

\section{Regularized Problems}
We turn now to regularized optimization in an inner-product space:
\begin{equation} \label{eq:F}
	\min_{x} \, F(x) \coloneqq f(x) + \psi(x),
\end{equation}
where both terms are convex, $f$ is $L$-Lipschitz-continuously
differentiable, and $\psi$ is extended-valued, proper, and closed, but
possibly nondifferentiable.  We also assume that $\psi$ is such that the
prox-operator can be applied easily, by solving the following
problem for any given $y$ and any $\lambda>0$:
\begin{equation*}
	\min_x\, \psi\left( x \right) + \frac{1}{2 \lambda} \nabs{x - y}^2.
	\label{eq:prox}
\end{equation*}
We assume further that the solution set $\Omega$ of \eqref{eq:F} is
nonempty, and denote by $F^*$ the value of $F$ for all $x\in\Omega$.
We discuss two algorithms to show how our techniques can
be extended to regularized problems.  They are proximal gradient (both
with and without line search) and stochastic proximal coordinate
descent with arbitrary sampling.

\subsection{Short-Step Proximal Gradient}
\label{subsec:shortprox}
Given $\bar L \geq L$, the $k$th step of the proximal gradient
algorithm is defined as follows:
\begin{equation}
	x_{k+1} \leftarrow x_k + d_k,\quad d_k \coloneqq \arg\min_{d}\,
	\nabla f(x_k)^\top
	d + \frac{\bar{L}}{2}\|d\|^2 + \psi\left( x_k+d \right).
	\label{eq:proxgrad}
\end{equation}
Note that $d_k$ is uniquely defined here, since the subproblem is
strongly convex.  It is shown in \cite{BecT09a,Nes13a} that $F(x_k)$
converges to $F^*$ at a rate of $O(1/k)$ for this algorithm, under our
assumptions.  We prove that a $o(1/k)$ rate can be attained.
\begin{theorem}
\label{thm:proxgrad}
Consider \eqref{eq:F} with $f$ convex and $L$-Lipschitz continuously
differentiable, $\psi$ convex, and nonempty solution set $\Omega$.
Given any $\bar L \geq L$, the proximal gradient method
\eqref{eq:proxgrad} generates iterates whose objective value converges
to $F^*$ at a $o(1/k)$ rate.
\end{theorem}
\begin{proof}
The method \eqref{eq:proxgrad} can be shown to be a descent method
from the Lipschitz continuity of $\nabla f$ and the fact that
$\bar L \geq L$.  From the optimality of the solution to
\eqref{eq:proxgrad} and that $x_{k+1} = x_k + d_k$,
\begin{equation}
-\left(\nabla f(x_k) + \bar L d_k\right) \in \partial \psi\left(
x_{k+1} \right),
\label{eq:opt}
\end{equation}
where $\partial \psi$ denotes the subdifferential of $\psi$.  Consider
any $\bx \in \Omega$.  We have from \eqref{eq:proxgrad} that for any
$T \geq 0$, the following chain of relationships holds:
\begin{align}
\nonumber
&~\nabs{x_{T+1} - \bx}^2 - \nabs{x_T - \bx}^2\\
\nonumber
= &~ 2 d_T^\top \left( x_T  - \bx
\right) + \nabs{d_T}^2\\
\nonumber
=&~ 2 d_T^\top \left( x_T  +d_T - \bx\right) -
\nabs{d_T}^2\\
\nonumber
=&~ 2 \left(d_T + \frac{\gT}{\bar{L}}\right)^\top
\left( x_{T+1} - \bx\right) - \frac{2}{\bar{L}}  \gT^\top \left( x_T + d_T
- \bx \right) - \nabs{d_T}^2\\
\nonumber
\stackrel{\eqref{eq:opt}}{\leq}&~ 2 \frac{\psi\left( \bx \right) -
	\psi\left( x_{T+1} \right)}{\bar L} - \frac{2}{\bar L} \gT^\top
	\left( x_T - \bx \right) - \frac{2}{\bar L} \gT^\top d_T -
	\nabs{d_T}^2\\
\nonumber
\leq&~ \frac{2}{\bar L}\left(\left(\psi\left( \bx \right) - \psi\left(
	x_{T+1} \right)\right) + f\left( \bx \right) - \left( f\left( x_T
	\right) + \gT^\top d_T + \frac{\bar L \nabs{d_T}^2}{2}\right)\right)\\
\leq&~ \frac{2 \left(F^* - F\left( x_{T+1} \right)\right)}{\bar L},
\label{eq:rkprox}
\end{align}
where in the last inequality, we have used
\begin{equation} \label{eq:hs9}
  f(x+d) \leq f(x) + \nabla f(x)^\top d + \frac{L}{2} \|d\|^2 \leq
  f(x) + \nabla f(x)^\top d + \frac{\bar{L}}{2} \|d\|^2.
\end{equation}
By rearranging \eqref{eq:rkprox} we obtain
\[
F(x_{T+1}) - F^* \le \frac{\bar{L}}{2} \left( \| x_T - \bar{x} \|^2 - \| x_{T+1} - \bar{x} \|^2 \right).
\]
The result follows by summing both sides of this expression over
$T=0,1,\dotsc,k-1$ and applying Lemma~\ref{lemma:techniques}.  \qed
\end{proof}

\subsection{Proximal Gradient with Line Search}
\label{subsec:proxlinesearch}
We discuss a line-search variant of proximal gradient, where the
update is defined as follows:
\begin{equation}
	x_{k+1} \leftarrow x_k + d_k,\quad d_k \coloneqq \arg\min_{d}\,
	\nabla f(x_k)^\top
	d + \frac{1}{2\alpha_k}\|d\|^2 + \psi\left( x_k+d \right),
	\label{eq:sparsa}
\end{equation}
where $\alpha_k$ is chosen such that for given $\gamma \in (0,1]$ and 
$C_1 \ge C_2 > 0$  defined as in \eqref{eq:bdd}, we have
\begin{equation}
	\alpha_k \in [C_2, C_1],\quad
	F\left( x_k + d_k \right) \leq F\left( x_k \right) -
	\frac{\gamma}{2 \alpha_k}\|d_k\|^2.
	\label{eq:alpha}
\end{equation}
This framework is a generalization of that in Section
\ref{subsec:linesearch}, and includes the SpaRSA algorithm
of \cite{WriNF09a}, which obtains an initial choice of $\alpha_k$ from a
Barzilai-Borwein approach and adjusts it until \eqref{eq:alpha} holds.
The approach of the previous subsection can also be seen as a special
case of \eqref{eq:sparsa}-\eqref{eq:alpha} through the following
elementary result, whose proof is omitted.
\begin{lemma}
\label{lemma:suff}
Consider a convex function $\psi$, a positive scalar $a >0$ and two
vectors $b$ and $x$. If $d$ is the unique solution of the strictly
convex problem
\begin{equation*}
\min_d\, b^\top d + \frac{a}{2}  \nabs{d}^2 + \psi(x+d),
\end{equation*}
then
\begin{equation}
	b^\top d + \frac{a}{2} \nabs{d}^2 + \psi(x+d) - \psi(x) \leq
	-\frac{a}{2} \nabs{d}^2.
	\label{eq:suff}
\end{equation}
\end{lemma}
By setting $b = \nabla f(x)$, $1/\alpha_k \equiv a = \bar L > 0$
(where $\bar{L} \ge L$), this lemma together with \eqref{eq:hs9}
implies that \eqref{eq:alpha} holds for any $\gamma \in (0,1]$.
Moreover, it also implies that for any $k \ge 0$,
\begin{align*}
	F\left( x_{k+1} \right) - F\left( x_k \right)
	= &~ f\left( x_k + d_k \right) - f\left( x_k \right) + \psi\left( x_k + d_k \right) -
	\psi \left( x_k \right)
	\\
	\stackrel{\eqref{eq:hs9}}{\leq}
	&~ \nabla f\left( x_k \right)^\top d_k +
	\frac{L}{2}\left\|d_k\right\|^2 + \psi\left( x_k + d_k \right) -
	\psi \left( x_k \right)\\
	=&~ \nabla f\left( x_k \right)^\top d_k +
	\frac{1}{2\alpha_k}\left\|d_k\right\|^2 + \psi\left( x_k + d_k
	\right) - \psi \left( x_k \right)
	+ \left(\frac{L}{2} -
	\frac{1}{2\alpha_k}\right) \left\|d_k\right\|^2\\
	\stackrel{\eqref{eq:suff}}{\le}&~ - \frac{1}{2\alpha_k}
	\left\|d_k\right\|^2 +
	\left(\frac{L}{2} -
	\frac{1}{2\alpha_k}\right)\left\|d_k\right\|^2\\
	=& -\left(\frac{1}{\alpha_k} - \frac{L}{2}
	\right)\left\|d_k\right\|^2.
\end{align*}
Therefore, for any $\gamma \in (0,1]$, \eqref{eq:alpha} holds
whenever 
\[
	\alpha > 0,
	-\frac{\gamma}{2 \alpha_k} \geq - \left(\frac{1}{\alpha_k} -
	\frac{L}{2}\right),
\]
or equivalently $$\alpha_k \in \left(0,\frac{2 - \gamma}{L}\right],$$ which is how the
upper bound for $C_2$ is set.

We show now that this approach also has a $o(1/k)$ convergence rate
on convex problems.
\begin{theorem}
\label{thm:proxline}
Consider \eqref{eq:F} with $f$ convex and $L$-Lipschitz continuously
differentiable, $\psi$ convex, and nonempty solution set $\Omega$.
Given some $\gamma \in (0,1]$ and $C_2$ and $C_1$
such that $C_1 \geq C_2$ and $C_2 \in (0, (2 - \gamma)/L]$, then the
algorithm \eqref{eq:sparsa} with $\alpha_k$ satisfying
\eqref{eq:alpha} generates iterates $\{ x_k \}$ whose objective values
converge to $F^*$ at a rate of $o(1/k)$. Moreover, the
sequence of iterates  is bounded.
\end{theorem}
\begin{proof}
From the optimality conditions of \eqref{eq:sparsa}, we have
\begin{equation}
-\left(\nabla f(x_T) + \frac{1}{\alpha_T} d_T \right)\in \partial \psi\left(
x_{T+1} \right).
\label{eq:opt2}
\end{equation}
Now consider any $\bx \in \Omega$.  We have from \eqref{eq:sparsa}
that for any $T \geq 0$, the following chain of relationships holds:
\begin{align}
\nonumber
&~\nabs{x_{T+1} - \bx}^2 - \nabs{x_T - \bx}^2\\
\nonumber
=&~ 2 d_T^\top \left( x_T  +d_T - \bx\right) -
\nabs{d_T}^2\\
\nonumber
=&~ 2 \left(d_T + \alpha_T \gT\right)^\top
\left( x_{T+1} - \bx\right) - {2}{\alpha_T} \gT^\top \left( x_T + d_T
- \bx \right) - \nabs{d_T}^2\\
\nonumber
\stackrel{\eqref{eq:opt2}}{\leq}&~ 2 \alpha_T \left( {\psi\left( \bx \right) - \psi\left( x_{T+1}
\right)} \right) - {2}{\alpha_T} \gT^\top \left( x_T - \bx \right) -
{2}{\alpha_T} \gT^\top d_T - \|d_T \|^2 \\
\nonumber
\leq&~ 2 \alpha_T \left( {\psi\left( \bx \right) - \psi\left( x_{T+1}
\right)} \right) - {2}{\alpha_T} \gT^\top \left( x_T - \bx \right) -
{2}{\alpha_T} \gT^\top d_T\\
\nonumber
= &~ 2 \alpha_T \left( {\psi\left( \bx \right) - \psi\left( x_{T+1}
\right)} \right) - {2}{\alpha_T} \gT^\top \left( x_T - \bx \right) -
{2}{\alpha_T} \gT^\top d_T +
 \alpha_T L \nabs{d_T}^2 - \alpha_T L \nabs{d_T}^2\\
\nonumber
\leq&~ 2\alpha_T \left(\psi\left( \bx \right) - \psi\left(
	x_{T+1}\right) + f\left( \bx \right) - \left(f\left( x_T \right) +
	\gT^\top d_T + \frac{L}{2}\nabs{d_T}^2\right) \right)+
         \alpha_T L \nabs{d_T}^2\\
\nonumber
\stackrel{\eqref{eq:alpha}}{\leq}&~ {2 \alpha_T \left(F^* - F\left(
	x_{T+1} \right)\right)} + \frac{2 L \alpha_T^2}{\gamma} \left(
	F(x_T) - F(x_{T+1}) \right)\\
\leq&~{2}{C_2}\left( F^* - F\left( x_{T+1} \right)
	\right) + \frac{2LC_1^2}{\gamma} \left( F\left( x_T \right)
	- F\left( x_{T+1} \right) \right).
\label{eq:rkprox2}
\end{align}
By rearrangement, of this inequality, we obtain
\begin{align*}
   F(x_{T+1})-F^* 
    \le \frac{L C_1^2}{\gamma C_2}(F(x_T)-F(x_{T+1}))
   + \frac{1}{2 C_2} \left( \| x_{T}-\bar{x}\|^2 - \| x_{T+1}-\bar{x}\|^2\right),
\end{align*}
and by summing both sides and using telescoping sums, we find that
$\sum_{T=0}^\infty (F(x_{T+1})-F^*) < \infty$, thus the conditions of
Lemma~\ref{lemma:techniques} are satisfied by $\Delta_T :=
F(x_T)-F^*$, and the $o(1/k)$ rate follows.

  By summing the inequality above finitely over $T=0,1,\dotsc,k-1$, we obtain
  \begin{align*}
    0 \le \sum_{T=0}^{k-1} (F(x_{T+1})-F^*) \le \frac{L
	C_1^2}{\gamma C_2} (F(x_0)-F^*)
     + \frac{1}{2 C_2} \left(  \|x_0 - \bar{x}\|^2 -  \|x_k - \bar{x}\|^2 \right).
  \end{align*}
  By rearranging this inequality, we obtain a uniform upper bound on
  $\|x_k - \bar{x} \|$, thus showing that the sequence $\{ x_k \}$ is
  bounded.
  \qed
\end{proof}


\subsection{Proximal Coordinate Descent}
We now discuss the extension of coordinate descent to \eqref{eq:F},
with the assumption \eqref{eq:Ls} on $f$, Euclidean domain of dimension $n$,
sampling weighted according to \eqref{eq:prob} as in
Section~\ref{subsec:cd}, and the additional assumption of separability
of the regularizer $\psi$, that is,
\begin{equation}
\psi(x) = \sum_{i=1}^n \psi_i(x_i),
\label{eq:separable}
\end{equation}
where each $\psi_i$ is convex, extended valued, and possibly
nondifferentiable.  As in our discussion of Section~\ref{subsec:cd},
the results in this subsection can be extended directly to the case of
block-coordinate descent.

Given the component-wise Lipschitz constants $L_1,L_2,\dotsc,L_n$ and
algorithmic parameters $\bar L_1, \bar{L}_2,\dotsc,\bar L_n$ with $\bar
L_i \geq L_i$ for all $i$, proximal coordinate descent updates have
the form
\begin{equation}
x_{k+1} \leftarrow x_k + d^k_{i_k} e_{i_k},\quad
d^k_{i_k} \coloneqq \arg\min_{d \in \Re}\, \nabla_{i_k} f(x_k) d +
	\frac{\bar{L}_{i_k}}{2}d^2 + \psi_{i_k}\left( (x_k)_{i_k} +
	d\right).
\label{eq:proxcd}
\end{equation}
With $p_i \equiv 1 / n$ for all $i$,
\cite{LuX15a} showed that the expected objective value converges to
$F^*$ at a $O(1/k)$ rate.  When arbitrary sampling \eqref{eq:prob} is
considered, \eqref{eq:proxcd} is a special case of the general
algorithmic framework described in \cite{LeeW18b}.  The latter paper
shows the same $O(1/k)$ rate for convex problems under the additional
assumption that for any $x_0$, we have
\begin{equation}
	\max_{x: F\left( x \right) \leq F\left( x_0 \right)}\, \dist\left(
	x, \Omega \right) < \infty.
	\label{eq:bddset}
\end{equation}

We show here that with arbitrary sampling according to
\eqref{eq:prob}, \eqref{eq:proxcd} produces $o(1/k)$ convergence rates
for the expected objective on convex problems, without the assumption
\eqref{eq:bddset}.

The following result makes use of the quantity $r_k$ defined in
\eqref{eq:r}.
\begin{theorem}
	\label{thm:proxcd}
  Consider \eqref{eq:F} with $f$ and $\psi$ convex and nonempty
  solution set $\Omega$. Assume further that
  \eqref{eq:separable} is true, and that \eqref{eq:Ls}
  holds with some $L_1,L_2,\dotsc,L_n > 0$.  Given $\{\bar
  L_i\}_{i=1}^n$ with $\bar L_i \geq L_i$ for all $i$, suppose that
  proximal coordinate descent defines iterates according to
  \eqref{eq:proxcd}, with $i_k$ chosen i.i.d.  according to a
  probability distribution satisfying \eqref{eq:prob}. Then
  $\E_{i_0,i_1,\dotsc,i_{k-1}}[F(x_k)]$ converges to $F^*$ at an
  asymptotic rate of $o(1/k)$.
  Moreover, given any $\bx \in \Omega$, the sequence of $\E_{i_0,\dotsc, i_{k-1}}
  r_k^2$ is bounded.
\end{theorem}
\begin{proof}
From \eqref{eq:Ls}, we first notice that in the update
\eqref{eq:proxcd},
\begin{align}
F\left( x_k + d^k_{i_k} e_{i_k} \right)  -  F\left( x_k \right) \leq &~
	\nabla_{i_k} f(x_k) d^k_{i_k} +
	\frac{\bar{L}_{i_k}}{2} \left(d^k_{i_k}\right)^2+
 \psi_{i_k}\left( \left(x_k\right)_{i_k} + d^k_{i_k}\right) -
	\psi_{i_k}\left( \left(x_k\right)_{i_k} \right).
\label{eq:suffCD2}
\end{align}
From Lemma~\ref{lemma:suff}, the method defined by \eqref{eq:proxcd}
is a descent method.  Optimality  of the subproblem in
\eqref{eq:proxcd} yields
\begin{equation}
-\left(\nabla_{i_T} f\left( x_T \right) + \bar L_{i_T} d^T_{i_T}
\right) \in \partial \psi_{i_T} \left( \left( x_T \right)_{i_T} +
d^T_{i_T} \right).
\label{eq:partial}
\end{equation}
By taking any $\bx \in \Omega$, and using the definition \eqref{eq:r}, we
have:
\begin{align}
\nonumber
r_{T+1}^2
=&~ r_T^2 + \frac{2 \bar L_{i_T}}{p_{i_T}} \left(d_{i_T}^\top\right)^\top
\left( x_T + d_{i_T}^T - \bx \right)_{i_T} - \frac{\bar
	L_{i_T}}{p_{i_T}} \left(d_{i_T}^T\right)^2\\
\nonumber
=&~ r_T^2 + \frac{2}{p_{i_T}} \left(\nabla_{i_T} f\left( x_T
	\right) + \bar L_{i_T}d_{i_T}^T\right)^\top \left( x_T +
	d_{i_T}^T - \bx \right)_{i_T}
	- \frac{\bar L_{i_T}}{p_{i_T}} \left( d_{i_T}^T \right)^2 \\
\nonumber
&\quad - \frac{2}{p_{i_T}} \nabla_{i_T} f\left(
	x_T \right)^\top \left( x_T - \bx \right)_{i_T} - \frac{2}{p_{i_T}}
	\nabla_{i_T} f\left( x_T \right)^\top d^T_{i_T}\\
\nonumber
\stackrel{\eqref{eq:partial}}{\leq}&~ r_T^2 + \frac{2}{p_{i_T}}
	\left(\psi_{i_T} \left( \bx_{i_T} \right) - \psi_{i_T} \left(
	\left( x_T\right)_{i_T} + d^T_{i_T} \right)  - \nabla_{i_T}
	f\left( x_T \right)^\top \left( x_T - \bx \right)_{i_T}\right)\\
\nonumber
&\quad -\frac{2}{p_{i_T}} \left(\nabla_{i_T} f\left( x_T \right)^\top
	d^T_{i_T} + \frac{\bar L_{i_T}}{2} \nabs{d_{i_T}^k}^2 \right)\\
\label{eq:rkcd}
\leq&~ r_T^2 + \frac{2}{p_{i_T}} \left(\psi_{i_T} \left( \bx_{i_T}
	\right) - \psi_{i_T} \left(\left(x_T\right)_{i_T} \right) -
	\nabla_{i_T} f\left( x_T \right)^\top \left(x_T - \bx
	\right)_{i_T}\right)\\
\nonumber
&\quad -\frac{2}{p_{i_T}} \left(\nabla_{i_T} f\left( x_T \right)^\top
	d^T_{i_T} + \frac{\bar L_{i_T}}{2} \nabs{d_{i_T}^T}^2 +
	\psi_{i_T} \left( \left( x_T\right)_{i_T} + d^T_{i_T} \right) -
	\psi_{i_T} \left(\left(x_T\right)_{i_T} \right) \right).
\end{align}
By taking expectation over $i_T$ on both sides of \eqref{eq:rkcd} and
using the convexity of $f$ together with \eqref{eq:suffCD2}, we obtain
\begin{subequations}
\begin{align}
\nonumber
&~\E_{i_T}\left[ r_{T+1}^2 \right] - r_T^2\\
\nonumber
\leq &~ 2\left( \psi\left( \bx \right) - \psi\left(x_T\right) +
	f\left( \bx \right) - f\left( x_T \right)\right)
+ 2 \left( \sum_{i=1}^n F\left( x_T \right) - F\left( x_T + d^T_i e_i
\right) \right)\\
\label{eq:toexplain}
\leq&~ 2 \left( F^* - F\left( x_T \right) \right)
+ \frac{2}{\pmin} \sum_{i=1}^n p_i \left( F\left( x_T
\right) -  F\left( x_T + d^T_i e_i \right) \right)\\
\label{eq:tosumcd3}
=&~ 2 \left( F^* - F\left( x_T \right) \right) + \frac{2}{\pmin} \left(F\left( x_T \right) - \E_{i_T} \left[F\left( x_{T+1}
	\right)\right]\right),
\end{align}
\end{subequations}
where in \eqref{eq:toexplain} we used the fact that \eqref{eq:proxcd}
is a descent method.  By taking expectation over $i_0,\dotsc,i_k$
on \eqref{eq:tosumcd3}, summing over $T=0,\dotsc,k$, and applying
Lemma~\ref{lemma:techniques}, we obtain the result.

Boundedness of $\E_{i_0,\dotsc,i_{k-1}} [r_k^2]$ follows from the same
telescoping sum and the fact that $F(x_k)$ decreases monotonically
with $k$.
\qed
\end{proof}

Our result shows that, similar to gradient descent and proximal
gradient, proximal coordinate descent and coordinate descent also
provide a form of implicit regularization in that the expected value
of $r_k$ is bounded.  Since $r_k$ can be viewed as a weighted
Euclidean norm, this observation implies that the iterates are also in
a sense expected to lie within a bounded region.

Our analysis here improves the rates in
\cite{LuX15a,LeeW18b} in terms of the dependency on $k$ and removes
the assumption of \eqref{eq:bdd} in \cite{LeeW18b}.  Even aside from
the improvement from $O(1/k)$ to $o(1/k)$, Theorem~\ref{thm:proxcd} is
the first time that a convergence rate for proximal stochastic
coordinate descent with arbitrary sampling for the coordinates is
proven without additional assumptions such as \eqref{eq:bddset}.  By
manipulating \eqref{eq:tosumcd3}, one can also observe how different
probability distributions affect the upper bound, and it might also be
possible to get better upper bounds by using norms different from
\eqref{eq:r}.

\section{Tightness of the $o(1/k)$ Estimate}
We demonstrate that the $o(1/k)$ estimate of convergence of $\{f(x_k)
\}$ is tight by showing that for any $\epsilon \in (0,1]$, there is a
  convex smooth function for which the sequence of function values
  generated by gradient descent with a fixed step size converges
  slower than $O(1/k^{1+\epsilon})$.  The example problem we provide
  is a simple one-dimensional function, so it serves also as a special
  case of stochastic coordinate descent and the proximal methods
  (where $\psi \equiv 0$) as well.  Thus, this example shows tightness
  of our analysis for all methods without line search considered in
  this paper.

Consider the one-dimensional real convex function
\begin{equation}
f(x) = x^p,
\label{eq:xp}
\end{equation}
where
$p$ is an even integer greater than $2$. The minimizer of this
function is clearly at
$x^*=0$, for which $f(0)=f^*=0$. Suppose that the gradient descent method
is applied starting from $x_0=1$. For any descent method, the
iterates $x_k$ are confined to $[-1,1]$ and we have
\[
\| \nabla^2 f(x) \| \le p(p-1)  \;\; \mbox{for all $x$ with
  $|x| \le 1$,}
\]
so we set $L = p(p-1)$. Suppose that $\bar\alpha \in
(0,2/L)$ as above. Then the iteration formula is
\begin{equation}
\label{eq:update}
x_{k+1} = x_k - \bar\alpha \nabla f(x_k) = x_k \left( 1- p \bar\alpha x_k^{p-2} \right),
\end{equation}
and by Lemma \ref{lemma:rk}, all iterates lie in a bounded set: the
level set $[-1,1]$ defined by $x_0$. In fact, since $p \ge 4$ and
$\bar\alpha \in (0,2/L)$, we have that
\begin{align*}
x_k \in (0,1] \; \Rightarrow \; 1-p\bar\alpha x_k^{p-2} &  \in \left(
  1-\frac{2p}{p(p-1)} x_k^{p-2}, 1 \right)
  \subseteq \left( 1-\frac{2}{p-1},1 \right) \subseteq \left( \frac23,
  1 \right),
\end{align*}
so that $x_{k+1} \in \left(\tfrac23 x_k,x_k \right)$ and the value of $L$
remains valid for all iterates.

We show by an informal argument that there exists a constant $C$ such that
\begin{equation} \label{eq:jw0}
  f(x_k) \approx \frac{C}{k^{p/(p-2)}}, \quad \mbox{for all $k$ sufficiently large.}
\end{equation}
From \eqref{eq:update} we have
\begin{equation} \label{eq:yj8}
f(x_{k+1}) = x_{k+1}^p = x_k^p \left( 1- p \bar\alpha x_k^{p-2} \right)^p =
f(x_k) \left( 1- p \bar\alpha f(x_k)^{(p-2)/p} \right)^p.
\end{equation}
By substituting the hypothesis \eqref{eq:jw0} into \eqref{eq:yj8}, and
taking $k$ to be large, we obtain the following sequence of equivalent
approximate equalities:
\begin{alignat*}{2}
 && \frac{C}{(k+1)^{p/(p-2)}}  & \approx \frac{C}{k^{p/(p-2)}}
  \left( 1- p \bar\alpha \frac{C^{(p-2)/p}}{k} \right)^p \\
 & \Leftrightarrow & \;\; \left( \frac{k}{k+1} \right)^{p/(p-2)} & \approx
  \left( 1- p \bar\alpha \frac{C^{(p-2)/p}}{k} \right)^p \\
 & \Leftrightarrow &\;\; \left( 1- \frac{1}{k+1} \right)^{p/(p-2)} & \approx
   1- p^2 \bar\alpha \frac{C^{(p-2)/p}}{k} \\
  & \Leftrightarrow & \;\;  1- \frac{p}{p-2} \frac{1}{k+1} & \approx
  1- p^2 \bar\alpha \frac{C^{(p-2)/p}}{k}
\end{alignat*}
This last expression is approximately satisfied for large $k$ if $C$
satisfies the expression
\[
\frac{p}{p-2} = p^2 \bar\alpha C^{(p-2)/p}.
\]

Stated another way, our result \eqref{eq:jw0} indicates that a
convergence rate faster than $O(1/k^{1+\epsilon})$ is not possible
when steepest descent with fixed steplength is applied to the function
$f(x) =x^p$ provided that
\[
\frac{p}{p-2} \le 1+\epsilon,
\]
that is,
\[
p \ge 2 \frac{1+\epsilon}{\epsilon} \;\; \makebox{and $p$ is a positive even integer.}
\]

We follow \cite{AttCPR18a} to provide a continuous-time analysis of
the same objective function, using a gradient flow argument.  For the
function $f$ defined by \eqref{eq:xp}, consider the following
differential equation:
\begin{equation}
	x'(t) = -\alpha \nabla f(x(t)).
	\label{eq:gradflow}
\end{equation}
Suppose that
\begin{equation}
	x(t) = t^{-\theta}
	\label{eq:xt}
\end{equation}
for some $\theta > 0$, which indicates that starting from any $t > 0$,
$x(t)$ lies in a bounded area.  Substituting \eqref{eq:xt} into
\eqref{eq:gradflow}, we obtain
\begin{equation*}
	-\theta t^{-\theta - 1} = -\alpha p t^{-\theta (p-1)},
\end{equation*}
which holds true if and only if the following equations are satisfied:
\begin{equation*}
	\begin{cases}
		\theta &= \alpha p,\\
		-\theta - 1 &= -\theta p + \theta,
	\end{cases}
\end{equation*}
from which we obtain
\begin{equation*}
	\begin{cases}
		\theta &= \frac{1}{p-2},\\
		\alpha &= \frac{1}{p(p-2)}.
	\end{cases}
\end{equation*}
Since $x$ decreases monotonically to zero, for all $ t \geq (p-1) / (p-2)$,
\[
	L = p \left( p-1 \right)
	\left(\frac{p -1}{p-2}\right)^{-\theta (p-2)} = p(p-2)
\]
is an appropriate value for a bound on $\| \nabla^2 f(x) \|$. These
values of $\alpha$ and $L$ satisfy $0<\alpha \leq \frac{1}{L}$, making
$\alpha$ a valid step size.
The objective value is $f(x(t)) = t^{-p / (p-2)}$, matching the rate
of \eqref{eq:jw0}.

\section*{Acknowledgment}
The authors thank Yixin Tao for a discussion that helped us to improve
the clarity of this work.
\bibliographystyle{spmpsci}      
\bibliography{gd}

\end{document}